\documentclass{amsart}

\swapnumbers \theoremstyle{plain}
\newtheorem{thm}{Theorem}[section]
\newtheorem{lem}[thm]{Lemma}

\newtheorem{lem-defn}[thm]{Lemma and definition}

\newtheorem{cor}[thm]{Corollary}
\newtheorem{prop}[thm]{Proposition}
\theoremstyle{definition}\newtheorem{rem}[thm]{Remark}
\newtheorem{defn}[thm]{Definition}

\theoremstyle{definition}\newtheorem{leer}[thm]{}

\newcommand{\Cal}{\mathcal}

\newcommand{\R}{\mathbb{R}}

\newcommand{\C}{\mathbb{C}}
\newcommand{\D}{\mathbb{D}}
\newcommand{\N}{\mathbb{N}}
\newcommand{\Z}{\mathbb{Z}}
\newcommand{\E}{\mathbb{E}}

\newcommand{\B}{\mathbb{B}}
\newcommand{\s}{\mathbb{S}}

 \DeclareMathOperator{\im}{Im}
\DeclareMathOperator{\ke}{Ker} \DeclareMathOperator{\rea}{Re}

 \numberwithin{equation}{section}

\usepackage{soul,cancel,url}
\usepackage{amsfonts}
\usepackage{hyperref}
\usepackage{color}
\usepackage{amsmath}
    \usepackage{amsxtra}
    \usepackage{amstext}
    \usepackage{amssymb}
   \usepackage{latexsym}
   \usepackage{graphicx}
\title {On the similarity of holomorphic matrices}
\author{J\"urgen Leiterer}
\address{Institut f\"ur Mathematik \\
Humboldt-Universit\"at zu Berlin \\Rudower Chaussee 25\\D-12489 Berlin , Germany}
\email{leiterer@mathematik.hu-berlin.de}

\thanks{Some of the results of this text are contained already in the preprint ``Local and global similarity of holomorphic
matrices'' sent to some colleagues in September 2016 and later posted in the arXiv \cite{Le1} (partially with different proofs).}

\thanks{MSC 2010: 32A99, 47A56, 15A21.}

\thanks{Keywords: holomorphic matrices, similarity, Oka principle}
\date{}
\dedicatory{Dedicated to the memory of Gennadi Henkin, teacher and coauthor}

\begin{document}

\begin{abstract} R. Guralnick \cite{Gu} proved that two holomorphic matrices on a noncompact connected Riemann surface, which are locally holomorphically similar, are globally holomorphically similar. We generalize this to  (possibly, non-smooth) one-dimensional Stein spaces. For Stein spaces of arbitrary dimension, we prove that global $\Cal C^\infty$ similarity implies global holomorphic similarity, whereas global continuous similarity is not sufficient.
\end{abstract}

\maketitle

\section{Introduction}

Let $X$ be a  complex space, e.g., a complex manifold or an analytic subset of a complex manifold. Denote by  $\mathrm{Mat}(n\times n,\C)$ the algebra of complex $n\times n$ matrices, and by $\mathrm{GL}(n,\C)$ the group of its invertible elements.

 \begin{defn}\label{7.6.16-} Two holomorphic maps $A,B:X\to \mathrm{Mat}(n\times n,\C)$ are called
(globally) {\bf  holomorphically  similar on $ X$} if there is a holomorphic  map $H:X\to\mathrm{GL}(n,\C)$ with $B=H^{-1}AH$ on $X$.
They are called
{\bf locally holomorphically similar at a point $\xi\in X$}  if there is a neighborhood $U$ of $\xi$ such that $A\vert_U$ and $B\vert_U$ are holomorphically similar on $U$.
Correspondingly,   {\bf continuous} and {\bf $\Cal C^k$ similarity} are defined.
\end{defn}

R. Guralnick \cite{Gu} proved the following

\begin{thm}\label{24.10.16+} Suppose  $X$ is a noncompact connected Riemann surface. Then, any two holomorphic maps $A,B:X\to  \mathrm{Mat}(n\times n,\C)$, which are locally holomorphically similar at each point of $X$, are globally holomorphically similar on $X$.
\end{thm}

Actually, Guralnick proved a more general theorem for matrices with elements in a Bezout ring (with some extra properties), and then  applies this to the ring of holomorphic functions on a non-compact connected Riemann surface. The ring of holomorphic functions on an arbitrary (non-smooth) one-dimensional Stein space is not Bezout. Therefore, it seems  that Guralnick's proof of Theorem \ref{24.10.16+} cannot be generalized to the non-smooth case, at least not in a straightforward way.

Nevertheless, in the present paper, we use Guralnick's result to prove

\begin{thm}\label{18.10.16} Suppose $X$ is a one-dimensional Stein space (possibly not smooth), and let $A,B:X\to  \mathrm{Mat}(n\times n,\C)$ be two holomorphic maps, which are locally holomorphically similar at each point of $X$. Then $A$ and $B$ are globally holomorphically similar on $X$.
\end{thm}

In the proof (given in Section \ref{28.10.16'}) we take advantage of the fact that the normalization  of $X$ is a Riemann surface each connected component of which  is noncompact, and there, we can apply Guralnick's result. Then we use the Oka principle for Oka pairs of  Forster and Ramspott \cite{FR1} (Proposition \ref{28.6.16-} below) to ``push down'' this to $X$.

The Oka principle of Forster and Ramspott is valid also for Stein spaces of arbitrary dimension. We deduce from it (in Section \ref{28.10.16}) the following Oka principle for the similarity of holomorphic matrices.

\begin{thm}\label{26.10.16} Suppose  $X$ is a Stein space (of arbitrary dimension), and let $A,B:X\to  \mathrm{Mat}(n\times n,\C)$ be  two holomorphic maps such that there exsists a continuous map $C:X\to \mathrm{GL}(n,\C)$ satisfying the following two conditions:
\begin{itemize}
\item[(a)] $B=C^{-1}AC$ on $X$.
\item[(b)] For each $\xi\in X$, there exist a neighborhood $U_\xi$ of $\xi$ and a holomorphic map $H_\xi:U_\xi\to \mathrm{GL}(n,\C)$ with $B=H_\xi^{-1}AH_\xi^{}$ on $U_\xi$ and $H_\xi(\xi)=C(\xi)$.
\end{itemize}
Then $A$ and $B$ are globally holomorphically similar on $X$.
\end{thm}
Note that conditions (a) and (b) imply that
\begin{itemize}
\item[(a')] $A$ and $B$ are globally continuously similar on $X$.
\item[(b')] $A$ and $B$ are locally holomorphically similar at each point of $X$.
\end{itemize}
However, conditions  (a') and (b') alone do not imply global holomorphic similarity. We show this by a counterexample (Theorem \ref{10.1.16'} below).

There are different criteria for local holomorphic similarity, which are known or can be easily obtained from known results. They are contained  in the following theorem (with invertible $\Phi$).

\begin{thm}\label{28.6.16**} Let $A,B:X\to \mathrm{Mat}(n\times n,\C)$ be holomorphic, $\xi\in X$ and $\Phi\in \mathrm{Mat}(n\times n,\C)$ such that $\Phi B(\xi)=A(\xi)\Phi$.
Suppose at least one of the following conditions is satisfied.
\begin{itemize}
\item[(i)] {\em (Wasow's criterion)} The dimension of the complex vector space
\begin{equation}\label{30.6.16}
\Big\{\Theta\in \mathrm{Mat}(n\times n,\C)\;\Big\vert\;\Theta B(\zeta)=A(\zeta)\Theta\Big\}
\end{equation}is constant for $\zeta$ in some neighborhood of $\xi$.
\item[(ii)]  {\em(Smith's criterion)} $\dim X=1$, $\xi$ is a smooth point of $X$,  and there exist a neighborhood $V_\xi$ of $\xi$ and a continuous map $C_\xi:V_\xi\to \mathrm{GL}(n,\C)$ such that $C_\xi B=AC_\xi$ on $V_\xi$, and $C_\xi(\xi)=\Phi$.
\item[(iii)] {\em (Spallek's criterion)} There exist a neighborhood $V_\xi$ of $\xi$ and a $\Cal C^\infty$ map $T_\xi:V_\xi\to \mathrm{Mat}(n\times n,\C)$ such that $T_\xi B=AT_\xi$ on $V_\xi$, and
$T_\xi(\xi)=\Phi$.
\end{itemize}
Then there exist a neighborhood $U_\xi$ of $\xi$ and a holomorphic $H_\xi:U_\xi\to \mathrm{Mat}(n\times n,\C)$ such that $H_\xi B=AH_\xi$ on $U_\xi$, and $H_\xi(\xi)=\Phi$.
\end{thm}

Proofs or  references (explaining also the role of the names `Wasow, Smith and Spallek') for the statements contained in this theorem will be given in  Section \ref{local}.

From Spallek's criterion it follows that each $\Cal C^\infty$ map $C:X\to \mathrm{GL}(n,\C)$ satisfying condition (a) in Theorem \ref{26.10.16} automatically also satisfies condition (b). Therefore, Theorem \ref{26.10.16} has  the following

\begin{cor}\label{26.10.16''} Suppose  $X$ is a Stein space (of arbitrary dimension). Let $A,B:X\to  \mathrm{Mat}(n\times n,\C)$ be  two holomorphic maps, which are globally $\Cal C^\infty$ similar on $X$. Then $A$ and $B$ are globally holomorphically similar on $X$.
\end{cor}

We show by an example (Theorem \ref{10.1.16'} below) that in this corollary $\Cal C^\infty $ cannot be replaced by $\C^k$ with $k<\infty$ (the same $k$ for all $A$, $B$).

Moreover, Spallek's criterion in particular says that local $\Cal C^\infty$ similarity at a point implies local holomorphic similarity at this point, and the Smith criterion says that, if $\dim X=1$, then, at the smooth points, merely local continuous similarity implies local holomorphic similarity. Therefore Theorem \ref{18.10.16} has the following

\begin{cor}\label{28.6.16***} Suppose $X$ is a  one-dimensional Stein space. Let $A,B:X\to \mathrm{Mat}(n\times n,\C)$ be holomorphic. Then, for global holomorphic similarity of $A$ and $B$ it is sufficient that, for each point $\xi\in X$, at least one of the following holds.

\vspace{2mm}
-- $A$ and $B$ are locally $\Cal C^\infty$ similar at $\xi$.

-- $\xi$ is a smooth point of $X$, and $A$ and $B$ are locally continuously similar at $\xi$.
\end{cor}

We show by examples (Theorem \ref{6.1.16''}) that in this corollary, at the non-smooth points, $\Cal C^\infty$ cannot be replaced by $\Cal C^k$ with $k<\infty$ (the same $k$ for all $A$, $B$ and $\xi$). However, see Remark \ref{3.9.16*}.

In \cite{Le1} and then, in  revised form, in \cite{Le2}, we  gave another proof of Theorem \ref{18.10.16}, which does not use Guralnick's Theorem \ref{24.10.16+} (and thereby is also a new proof for Guralnick's result), but which is much longer than the proof given here.  An advantage of this other proof is that it is not restricted to the one-dimensional case. For example, in a forthcoming paper, we will show that the claim of Theorem \ref{24.10.16+} remains valid if $X$ is a convex domain in $\C^2$ (for convex domains in $\C^3$, this is not true).

{\bf Acknowledgements:} I want to thank F. Forstneri\u c, F. Kutzschebauch and J. Ruppenthal for helpful discussions (in particular, see Remark \ref{3.9.16**})

\section{Notations}\label{20.6.16}

$\N$ is the set of natural numbers including $0$. $\N^*=\N\setminus\{0\}$.

If $n,m\in \N^*$, then by $\mathrm{Mat}(n\times m,\C)$ we denote the space of complex $n\times m$ matrices ($n$ rows and $m$ columns), and by
$\mathrm{GL}(n,\C)$ we denote the  group of invertible complex $n\times n$ matrices.

The unit matrix in $\mathrm{Mat}(n\times n,\C)$ will be denoted by $I_n$ or simply by $I$.

$\ke \Phi$ denotes the kernel,  $\im \Phi$ the image and $\Vert\Phi\Vert$ the operator norm of a matrix $\Phi\in \mathrm{Mat}(n\times m,\C)$ considered as a linear map between the Euclidean spaces $\C^m$ and $\C^n$.

By a complex space we always mean a {\em reduced} complex space \cite{GR} (which is the same as an {\em analytic space} in the terminology used, e.g., in \cite{C} and \cite{L}).

\section{Preparations concerning sheaves}

\vspace{2mm}
Let $X$ be a  topological  space,  and  $G$ a topological group (abelian or non-abelian).
Then we denote by $\Cal C^G_X$, or simply by $\Cal C^G$, the sheaf of continuous $G$-valued maps on $X$, i.e., $\Cal C^G$ is the map which assigns to each open $U\subseteq X$ the group
 $ \Cal C^{G}(U)$  of alll continuous maps $f:U\to {G}$ if $U\not=\emptyset$, and the group which consist only of the neutral element of $G$ if $U=\emptyset$.

All sheaves in this paper are subsheaves of $\Cal C_X^G$ (for some $X$ and some $G$), i.e., a map $\Cal F$ which assigns to each open $U\subseteq X$ a subgroup $\Cal F(U)$ of $\Cal C^G(U)$ such that:
\begin{itemize}
\item[--] If $V\subseteq U$ are non-empty open subsets of $X$, then, for each $f\in \Cal F(U)$, the restriction of $f$ to $V$, $f\vert_V$, belongs to $\Cal F(V)$.
\item[--] If $U\subseteq X$ is  open  and $f\in \Cal C^G(U)$ is  such that, for each $\xi\in U$, there is an
open neighborhood $V\subseteq U$ of $\xi$ with $f\vert_V\in \Cal F(V)$, then  $f\in \Cal F(U)$.
\end{itemize}
 If $\Cal F$ and $\Cal G$ are two subsheaves of $\Cal C^{G}_X$, then $\Cal F$ is called a subsheaf of $\Cal G$ if $\Cal F(U)\subseteq \Cal G(U)$ for all open $U\subseteq X$.

  If $X$ is  a complex space and  $G$ is a complex Lie group, then we denote by $\Cal O^{G}_X$, or simply by $\Cal O^{G}$,  the  subsheaf of $\Cal C^G_X$ which assigns to each non-empty open $U\subseteq X$, the group $\Cal O^G(U)$ of all holomorphic maps from $U$ to $G$.

\vspace{2mm}
Let $\Cal F$ be a subsheaf of $\Cal C^G$.

\vspace{2mm}
Let $\Cal U=\{U_i\}_{i\in I}$  an open covering of $X$.

A family $f_{ij}\in\Cal F(U_i\cap U_j)$, $i,j\in I$, is called a {\bf $(\Cal U,\Cal F)$-cocycle} if (with the group operation in $G$ written as a multiplication)
\begin{equation*}
f_{ij}f_{jk}=f_{ik}\quad\text{on}\quad U_i\cap U_j\cap U_k\quad\text{for all}\quad i,j,k\in I.\quad\footnote{Here and in the following we use the  convention that  statements like
``$f=g$ on $\emptyset$'' or ``$f:=g$ on $\emptyset$''
have to be omitted.}
\end{equation*}Note that then always $f^{-1}_{ij}=f^{}_{ji}$ and $f^{}_{ii}$ is identically equal to the neutral element of $G$.
The set of all $(\Cal U,\Cal F)$-cocycles will be denoted by $Z^1(\Cal U,\Cal F)$.
Two cocycles  $\{f_{ij}\}$ and $\{g_{ij}\}$ in $Z^1(\Cal U,\Cal F)$ are called {\bf $\Cal F$-equivalent} if there exists a family $h_i\in \Cal F(U_i)$, $i\in I$, such that
\[
f^{}_{ij}=h^{}_ig^{}_{ij}h^{-1}_j\quad\text{on}\quad U_i\cap U_j\quad\text{for all}\quad i,j\in I.
\] If, in this case, for all $i,j$, the map $g_{ij}$ is identically equal to the neutral element of $G$,  then $f$ is called {\bf $\Cal F$-trivial}.

\vspace{2mm}
We say that $f$ is an {\bf $\Cal F$-cocycle} (on $X$), if there exists an open covering $\Cal U$ of $X$ with $f\in Z^1(\Cal U,\Cal F)$. This covering then is called {\bf the covering of $f$}.
As usual we write
\[H^1(X,\Cal F)=0
\] to say that each $\Cal F$-cocycle is $\Cal F$-trivial.

\vspace{2mm}
Let $\Cal U=\{U_i\}_{i\in I}$ and $\Cal U^*=\{U^*_\alpha\}_{\alpha\in I^*}$  be two  open coverings of $X$ such that $\Cal U^*$ is a refinement of $\Cal U$, i.e., there is a map
$\tau:I^*\to I$ with $U^*_\alpha\subseteq U_{\tau(\alpha)}$ for all $\alpha\in I^*$.
Then we say  that
a $(\Cal U^*,\Cal F)$-cocycle $\{f^*_{\alpha}\}_{\alpha,\beta\in I^*}$  is {\bf induced} by a $(\Cal U,\Cal F)$-cocycle $\{f_{ij}\}_{i,j\in I}$
if this map $\tau$ can be chosen so that
\[
f^*_{\alpha\beta}=f_{\tau(\alpha)\tau(\beta)}^{}\quad\text{on}\quad U^*_i\cap U^*_j\quad\text{for all}\quad \alpha,\beta\in I^*.
\]
We need the following  well-known  proposition, see \cite[p. 101]{C} for ``if'' and  \cite[p. 41]{Hi} for ``only if''.
\begin{prop}\label{17.12.15--} Let $f,g\in Z^1(\Cal U,\Cal F)$ and $f^*,g^*\in Z^1(\Cal U^*,\Cal F)$ such that $f^*$ and $g^*$ are  induced by $f$ and $g$, respectively.
Then $f$ and $g$ are $\Cal F$-equivalent if and only if $f^*$ and $g^*$ are $\Cal F$-equivalent.
\end{prop}

Now let $\Cal U$ and $\Cal V$ be two arbitrary open coverings of $X$, $f\in Z^1(\Cal U,\Cal F)$ and $g\in Z^1(\Cal V,\Cal F)$.
Then we say that $f$ and $g$ are {\bf $\Cal F$-equivalent} if there exist an open covering $\Cal W$ of $X$, which is  a refinement of both $\Cal U$ and
$\Cal V$, and $(\Cal W,\Cal F)$ cocycles $f^*$ and $g^*$, which are induced by $f$ and $g$, respectively, such that $f^*$ and $g^*$ are $\Cal F$
equivalent. By Proposition \ref{17.12.15--}, this definition is in accordance with the definition of equivalence given above for $\Cal U=\Cal V$.

\section{An Oka principle and proof of Theorem \ref{26.10.16}}\label{28.10.16}

\begin{defn}\label{6.6.16''} Let  $\Phi\in\mathrm{Mat}(n\times n,\C)$.
We denote by $\mathrm{Com\,} \Phi$ the algebra of all $\Theta\in \Phi\in\mathrm{Mat}(n\times n,\C)$ with $\Phi \Theta=\Theta \Phi$, and by $\mathrm{GCom\,} \Phi$ we denote the group of invertible elements of $\mathrm{Com\,} \Phi$.
Note that, as easily seen,
\begin{align}&\label{17.8.16+}\mathrm{GCom\,} \Phi=\mathrm{GL}(n,\C)\cap \mathrm{Com\,} \Phi,\text{ and}\\
&\label{26.8.16+}
 \mathrm{Com\,}(\Gamma^{-1}\Phi\Gamma)=\Gamma^{-1}(\mathrm{Com \,}\Phi)\Gamma\quad\text{for all}\quad\Gamma\in\mathrm{GL}(n,\C).
 \end{align}
\end{defn}
\begin{lem}\label{19.1.16''}  $\mathrm{GCom\,} \Phi$  is connected, for each $\Phi\in\mathrm{Mat}(n\times n,\C)$.
 \end{lem}
\begin{proof}
Let $\Theta\in \mathrm{GCom\,} \Phi$. Since the set of eigenvalues of $\Phi$ is finite, and the numbers
$0$ and $-1-\Vert \Theta\Vert$ do not belong to it,  then we can find a continuous map
$\lambda:[0,1]\to \C$ such that $\lambda(0)=0$, $\lambda(1)= 1+\Vert\Theta\Vert$ and $\Theta+\lambda(t)I\in\mathrm{GL}(n,\C)$ for all $0\le t\le 1$.
Setting
\[
\gamma(t)=\begin{cases}\Theta+\lambda(t)I\quad&\text{if}\quad 0\le t\le 1,\\
(1+\Vert \Theta\Vert)\Big(\frac{2-t}{1+\Vert \Theta\Vert}\Theta+I\Big)\quad &\text{if}\quad 1\le t\le 2,\\
(1+(3-t)\Vert\Theta\Vert) I \quad &\text{if}\quad 2\le t\le 3,
\end{cases}
\]
then we obtain a continuous path $\gamma$ in $\mathrm{GL}(n,\C)$, which connects $\Theta=\gamma(0)$ with
$I=\gamma(3)$. Since $\Theta\in \mathrm{Com\,}\Phi$, from the definition of $\gamma$ it is clear that the values of $\gamma$ belong to the algebra $\mathrm{Com\,\Phi}$. In view of \eqref{17.8.16+}, this means that $\gamma$ lies inside $\mathrm{GCom\,}\Phi$.
\end{proof}

\begin{defn}\label{18.8.16n'}
Let $X$ be a complex space, and  $A:X\to \mathrm{Mat}(n\times n,\C)$ holomorphic. We introduce the families
\[
 \mathrm{Com\,}A:=\big\{\mathrm{Com\,}A(z)\big\}_{z\in X}\quad\text{and}\quad \mathrm{GCom\,}A:=\big\{\mathrm{GCom\,}A(z)\big\}_{z\in X}.
\]
If the dimension of  $\mathrm{Com\,}A(z)$ does not depend on $z$, then it follows from the Wasow criterion (Theorem \ref{28.6.16**}, condition (i)) that $\mathrm{Com\,}A$ is a holomorphic vector bundle, but it is clear that this dimension can jump (in an analytic set).  But even if $\mathrm{Com\,}A$ is a holomorphic vector bundle,   $\mathrm{Com\,}A$ need not be locally trivial as a bundle of algebras. In particular, $\mathrm{GCom\,}A$ need not be locally trivial as a bundle of groups. Moreover, it is possible that  $\mathrm{GCom\,}A$ is not locally trivial as a bundle of topological spaces. We give an example.
\end{defn}

\begin{leer}\label{24.8.16'}{\bf Example.}
Let $X=\C$ and $A(z):=\begin{pmatrix}z &1\\0&0\end{pmatrix}$, $z\in \C$.
Then
\[
\begin{pmatrix}a&b\\c&d\end{pmatrix}\in\mathrm{Com\,}A(z)\Leftrightarrow\begin{pmatrix}za+c &zb+d\\0&0\end{pmatrix}=\begin{pmatrix}za &a\\zc&c\end{pmatrix}\Leftrightarrow c=0\text{ and }a=zb+d,
\]which implies that $\dim \mathrm{Com\,}A(z)=2$ for all $z\in\C$. However
\[
\mathrm{GCom\,}A\,(0)=\bigg\{\begin{pmatrix}a&b\\0&a\end{pmatrix}\;\bigg\vert\; a\in \C^*, b\in \C\bigg\}
\]whereas, for $z\not=0$, $\mathrm{GCom\,}A\,(z)$ is isomorphic to
\[
\bigg\{\begin{pmatrix}a&0\\0&d\end{pmatrix}\;\bigg\vert\; a,d\in \C^*\bigg\}\quad\text{if}\quad z\not=0,
\]which implies that $\pi_1\big(\mathrm{GCom\,}A\,(0)\big)=\Z$ whereas
$\pi_1\big(\mathrm{GCom\,}A\,(z)\big)=\Z^2 $ if $ z\not=0$. Hence, for $z\not=0$, $\mathrm{GCom\,}A(z)$
is  not homeomorphic to $\mathrm{GCom\,}A(0)$.
\end{leer}

\begin{defn}\label{24.8.16''}  Let $X$ be a complex space, and  $A:X\to \mathrm{Mat}(n\times n,\C)$ holomorphic.
Even if the families $\mathrm{Com\,}A$ and/or $\mathrm{GCom\,}A$ are not locally trivial, their sheaves of holomorphic and  continuous sections  are well-defined. We denote them by $\Cal O^{\mathrm{Com\,}A}$, $\Cal O^{\mathrm{GCom\,}A}$, $\Cal C^{\mathrm{Com\,}A}$ and $\Cal C^{\mathrm{GCom\,}A}$, respectively.

Further, we define subsheaves
$\widehat{\Cal C}^{\mathrm{Com\,}A}$  and $\widehat{\Cal C}^{\mathrm{GCom\,}A}$ of $\Cal C^{\mathrm{Com\,}A}$  and $\Cal C^{\mathrm{GCom\,}A}$, respectively, as follows: if  $U$ is a non-empty open subset of $X$, then   $\widehat{\Cal C}^{\mathrm{Com\,}A}(U)$ is the algebra of all continuous maps $f:U\to \mathrm{Mat}(n\times n,\C)$
such that, for each $\xi\in U$, the following condition is satisfied:
\begin{equation}\label{19.8.16}\begin{cases}&\text{there exist a neighborhood }V_\xi\text{ of }\xi\\&\text{and }h_\xi\in \Cal O^{\mathrm{Com\,}A}(V_\xi)\text{ with }h(\xi)=f(\xi),
\end{cases}\end{equation}and we set $\widehat{\Cal C}^{\mathrm{GCom\,}A}(U)=\Cal C^{\mathrm{GL}(n,\C}(U)\cap\widehat{\Cal C}^{\mathrm{Com\,}A}(U)$.
\end{defn}

The following proposition is a special case of the Oka principle for Oka pairs of  O. Forster and K. J.  Ramspott \cite[Satz 1]{FR1}.
\begin{prop}\label{28.6.16-} Let $X$ be a Stein space, and $A:X\to \mathrm{Mat}(n\times n,\C)$  holomorphic. Then each $\widehat{\Cal C}^{\mathrm{GCom\,}A}$-trivial $\Cal O^{\mathrm{GCom\,}A}$-cocycle is ${\Cal O}^{\mathrm{GCom\,}A}$-trivial.
\end{prop}
Indeed, it is easy to see that, for each non-empty open $U\subseteq X$ we have: If $h\in \Cal O^{\mathrm{Com\,}A}(U)$, then $e^h\in\Cal O^{\mathrm{GCom\,}A}(U)$, and, if $ H\in\Cal O^{\mathrm{GCom\,}A}(U)$  with $\sup_{\zeta\in U}\Vert H(\zeta)-I\Vert<1 $, then
\[
\log H:=\sum_{\mu=1}^\infty (-1)^{\mu-1}\frac{(H-I)^\mu}{\mu}\;\in\; \Cal O^{\mathrm{Com\,}A}(U).
\]
This shows that $\Cal O^{\mathrm{GCom\,}A}$ is a {\em coherent $\Cal O$-subsheaf} of $\Cal O^{\mathrm{GL}(n,\C)}_X$ in the sense of \cite[\S 2]{FR1}, where $\Cal O^{\mathrm{Com\,}A}$ is the {\em generating sheaf of Lie algebras}.
Moreover, as observed in \cite[\S 2.3, example b)]{FR1}), the pair $\big(\Cal O^{\mathrm{GCom\,}A},\widehat{\Cal C}^{\mathrm{GCom\,}A}\big)$ is an {\em admissible pair} in the sense of \cite{FR1}, which, trivially, satisfies condition (PH) in Satz 1 of \cite{FR1}). Therefore the proposition follows from  that Satz 1. \qed

\begin{leer}\label{24.8.16*}{\em Proof of Theorem \ref{26.10.16}:} Since $A$ and $B$ are locally holomorphically similar at each point of $X$, we can find an open covering $\{U_i\}_{i\in I}$ of $X$ and holomorphic maps $H_i:U_i\to \mathrm{GL}(n,\C)$ such that
\begin{equation}\label{27.10.16}
B=H_i^{-1}AH_i^{}\quad\text{on}\quad U_i.
\end{equation}Then
$H_i^{-1}AH_i^{}=B=H_j^{-1}AH_j^{}$ on $U_i\cap U_j$. Hence $AH_i^{}H_j^{-1}=H_i^{}H_j^{-1}A$ on $U_i\cap U_j$, i.e.,
 the family
 \begin{equation}\label{24.8.16**}\{H_i^{}H^{-1}_j\}_{i,j\in I}
 \end{equation} is an $\Cal O^{\mathrm{GCom\,}A}$-cocycle.

Now, by hypothesis, we have a continuous map $C:X\to \mathrm{GL}(n,\C)$ satisfying conditions (a) and (b) in Theorem \ref{26.10.16}.  Set $c_i=H_iC$ on $U_i$.  We claim that
\begin{equation}\label{27.10.16'''}c_i\in\widehat{\Cal C}^{\mathrm{GCom\,}A}(U_i).
\end{equation}
 Indeed, let $\xi\in U_i$ be given. Then, by condition (b), we can find a neighborhood $V_\xi$ of $\xi$ and a holomorphic map $H_\xi:V_\xi\to \mathrm{GL}(n,\C)$ with
\begin{align}&\label{27.10.16'}B=H_\xi^{-1}AH_\xi\quad\text{on}\quad V_\xi,\quad\text{and}\\
&\label{27.10.16''}H_\xi(\xi)=C(\xi).
\end{align}Set $h=H_i^{}H_\xi^{-1}$ on $U_i\cap V_\xi$. Then  from \eqref{27.10.16'} and \eqref{27.10.16} it follows that
\[hAh^{-1}=H_i^{}H_\xi^{-1}AH_\xi^{}H_i^{-1}=H_i^{}BH_i^{-1}=A\quad\text{on}\quad U_i\cap V_\xi,
\]i.e., $h\in \Cal O^{\mathrm{GCom\,}A}(U_i\cap V_\xi)$, and from \eqref{27.10.16''} we see that
\[
h(\xi)=H_i^{}(\xi)C(\xi)^{-1}_{}=c_i^{}(\xi).
\]which proves \eqref{27.10.16'''}.

 Moreover
 \[
 c_i^{}c_j^{-1}=H_i^{}CC^{-1}_{}H_j^{-1}=H_i^{}H_j^{-1}\quad\text{on}\quad U_i\cap U_j.
 \]
Together with \eqref{27.10.16'''} this shows that the cocycle \eqref{24.8.16**} is  $\widehat{\Cal C}^{\mathrm{GCom\,}A}$-trivial. By Proposition \ref{28.6.16-}) this means that this cocycle is even  $\Cal O^{\mathrm{GCom\,}A}$-trivial, i.e., $H_i^{}H_j^{-1}=h_i^{}h_j^{-1}$ on $U_i\cap U_j$,
for some family $h_i\in \Cal O^{\mathrm{GCom\,}A}(U_i)$. Then $h_i^{-1}H_i^{}=h_j^{-1}H_j^{}$ on $U_i\cap U_j$. Hence, there is a well-defined  global holomorphic map $H:X\to \mathrm{GL}(n,\C)$ with
$H=H_i^{-1}h_i^{}$ on $U_i$,  and which satisfies, by \eqref{27.10.16}, $H^{-1}BH=h_i^{-1}H_i^{}BH_i^{-1}h_i^{}=h_i^{-1}Ah_i^{}=A$ on $X$.
\qed
\end{leer}

\section{Proof of Theorem \ref{18.10.16}{}}\label{28.10.16'}

\begin{lem}\label{25.6.16} Let $X$ be a complex space, $A:X\to \mathrm{Mat}(n\times n,\C)$ a holomorphic map,  $\Lambda\subseteq X$ a finite set,  $U$ a neighborhood of $\Lambda$ and $f\in \widehat{\Cal C}^{\mathrm{GCom\,}A}(U)$ {\em(Def. \ref{24.8.16''})}. Then there exists a neighborhood $W_1\subseteq U$ of $\Lambda$ such that, for each neighborhood $W_2$ of $\Lambda$ with $\overline W_2\subseteq W_1$, there is a map
$\widetilde f\in \widehat{\Cal C}^{\mathrm{GCom}(A)}(X)$ with $\widetilde f=f$ on $W_2$ and $\widetilde f=I$ on $X\setminus W_1$. \iffalse the following properties
\begin{itemize}
\item[(i)] $\widetilde f=f$ on $K$;
\item[(ii)]  $\widetilde f=I$ on $X\setminus V$;
\item[(iii)] there exists a continuous map $H:[0,1]\times X\to \mathrm{GL}(n,\C)$ such that
\begin{align*} &H(t,\cdot)\in \Cal C^{\mathrm{GCom}(A)}(X)\quad\text{for all}\quad 0\le t\le 1,\\
& H(t,\zeta)=I\quad \text{for all}\quad 0\le t\le 1\quad\text{and}\quad\zeta\in X\setminus U,\\
&H(0,\cdot)=\widetilde f\quad\text{and}\quad H(1,\zeta)=I\quad\text{for all}\quad \zeta\in X.
\end{align*}
\end{itemize}\fi
\end{lem}

\begin{proof} We may assume that $\Lambda$ consist only of one point, $\xi$. Choose  a neighborhood $V$ of $\xi$ so small that $\overline{V}$ is compact and contained in $U$, and set
\[
\alpha=1+\max_{\zeta\in \overline{V}}\Vert f(\zeta)\Vert.
\]Then $-\alpha$ is not an eigenvalue of $f(\xi)$. Moreover, $0$ is not an eigenvalue of $f(\xi)$ (as $f(\xi) $ is invertible).  Therefore, we can find a continuous function $\lambda:[-1,0]\to \C$ such that $\lambda(-1)=0$, $\lambda(0)= \alpha$ and $f(\xi)+\lambda(t)I\in \mathrm{GL}(n,\C)$ for all $-1\le t\le 0$. Since $f$ is continuous,  we can choose a neighborhood $W_1\subseteq V$ of $\xi$ so small that $\overline W_1\subseteq V$ and
\[
f(\zeta)+\lambda(t)I\in \mathrm{GL}(n,\C)\quad\text{for all}\quad -1\le t\le 0\text{ and  }\zeta\in \overline{W_1}.
\]Then, setting, for  $\zeta\in\overline W_1$,
\[
g(t,\zeta)=\begin{cases}f(\zeta)+\lambda(t)I\quad&\text{if}\quad -1\le t\le 0,\\
\Big(1-t+\frac{t}{\alpha}\Big)\Big(f(\zeta)+\alpha I\Big)&\text{if}\quad 0\le t\le 1,
\end{cases}
\]we obtain a continuous map  $g:[-1,1]\times \overline{W_1}\to \mathrm{GL}(n,\C)$ (recall that, by definition,  $\alpha\ge 1$ and therefore $1-t+\frac{t}{\alpha}\not=0$ for $0\le t\le 1$) such that
\begin{align*}&g(-1,\zeta)=f(\zeta)\quad\text{and}\quad g(1,\zeta)=\frac{1}{\alpha}f(\zeta)+I\quad\text{for all}\quad \zeta\in\overline{W_1},\text{ and},\\
& g(t,\cdot)\big\vert_{W_1}\in \widehat{\Cal C}^{\mathrm{GCom}(A)}(W_1)\quad\text{for all}\quad -1\le t\le 1.
\end{align*} Moreover, it follows from the definition of $\alpha$ that
\begin{equation*}\Vert g(1,\zeta)-I\Vert<1\quad\text{for all}\quad\zeta\in \overline{W_1}.
\end{equation*}
Choose  $-1=t_1<t_2<\ldots<t_m=1$ such that
\[
\big\Vert g(t_j,\zeta)g(t_{j+1},\zeta)^{-1}-I\big\Vert<1\quad\text{for all}\quad \zeta\in \overline{W_1}\text{ and }1\le j\le m-1,
\]and define, for $\zeta\in W_1$,
\[g_m^{}(\zeta)=g(1,\zeta)\quad\text{and}\quad
g_j^{}(\zeta)=g(t_j,\zeta)g(t_{j+1},\zeta)^{-1} \quad\text{if}\quad 1\le j\le m-1.\\
\]Then $g_j\in \widehat{\Cal C}^{\mathrm{GCom}(A)}(W_1)$ and $\Vert g_j-I\Vert<1$ on $W_1$, for  $1\le j\le m$, and
\[
f=g_1\cdot\ldots \cdot g_m\quad\text{on}\quad W_1.
\]
Now let  a neighborhood $W_2$ of $\xi$ with $\overline W_2\subseteq W_1$  be given. Choose a continuous function $\chi:X\to [0,1]$ with $\chi=1$ on $W_2$ and $\chi=0$ on $X\setminus W_1$. Then
\[\widetilde f(\zeta):=\begin{cases}\Big(I+\chi(\zeta)\big(g_1(\zeta)-I\big)\Big)\cdot\ldots\cdot \Big(I+\chi(\zeta)\big(g_m(\zeta)-I\big)\Big)\quad&\text{if}\quad\zeta\in W_1,\\
I&\text{if}\quad \zeta\in X\setminus W_1,
\end{cases}
\]has the required properties
\end{proof}

\begin{lem}\label{17.10.16--} Let $X$ be a one-dimensional Stein space, and $A:X\to \mathrm{Mat}(n\times n,\C)$ a holomorphic map. Then  $H^1\big(X,\Cal O^{\mathrm{GCom\,}A}\big)=0$.
\end{lem}
\begin{proof}
Let an $\Cal O^{\mathrm{GCom\,}A}$-cocycle $\{f_{ij}\}_{i,j\in I}$ with the covering  $\Cal U=\{ U_i\}_{i\in I}$ be given. We have to prove that this cocycle is $\Cal O^{\mathrm{GCom\,}A}$-trivial.

Denote by $S$ the set of non-smooth points of $X$.
Since $X$ is one-dimensional, $S$ is discrete and closed in $X$. It follows that $X$ admits arbitrary fine open coverings such that  each point of $S$ is contained in precisely one of the sets of the covering. Therefore, by Proposition \ref{17.12.15--},  we may assume that
\begin{equation}\label{25.6.16m'}
 S\cap U_i\cap U_j=\emptyset\quad\text{for all}\quad i,j\in I\text{ with }i\not=j,
\end{equation}
which implies  that, for each $\xi\in S$, there is precisely one index in $I$, $\tau(\xi)$, such that
$\xi\in U_{\tau(\xi)}$, and $\xi\not\in U_i$ if $i\not=\tau(\xi)$. Shrinking the sets $U_i$ with $i\in I\setminus \tau(S)$, we can moreover achieve that, for each $\xi\in S$, there is a neighborhood $V_\xi$ of $\xi$ with
\begin{align}&\label{25.6.16m''}
 V_\xi\subseteq U_{\tau(\xi)},\quad\text{and}\\
&\label{26.6.16+}V_\xi\cap U_i=\emptyset\quad \text{if}\quad i\not=\tau(\xi).
\end{align}

Now let $\pi:\widetilde X\to X$ be the normalization of $X$ (see, e.g., \cite[Ch. VI, \S 4]{L}).  Since, for each  $\xi\in S$, $\pi^{-1}(\xi)$ is finite and $S$ is discrete and closed in $X$, then
\[\widetilde S:=\pi^{-1}(S)
\]is discrete and closed in $\widetilde X$. Further let
$\widetilde A:=A\circ \pi$, $\widetilde U_i:=\pi^{-1}(U_i)$, and $\{\widetilde f_{ij}\}$ the $\Cal O^{\mathrm{GCom\,}\widetilde A}$-cocycle with the covering $\{\widetilde U_i\}_{i\in I}$ defined by
\[\widetilde f_{ij}:=f_{ij}\circ \pi\quad \text{on}\quad \widetilde U_i\cap \widetilde U_j.
\]
 The connected components of $\widetilde X$ are the normalizations of the irreducible components of $X$ (see, e.g., \cite[Ch. VI, \S 4.2]{L}).
 Since $X$ is one-dimensional, this implies by the Puiseux theorem (see, e.g., \cite[Ch. VI, \S 4.1]{L})  that the connected components of $\widetilde X$ are Riemann surfaces. Since $X$ is Stein and therefore non of the irreducible components of $X$ is compact, it follows that  each of the connected components of $\widetilde X$ is a non-compact connected Riemann surface.
Therefore, by Grauert's theorem \cite[Satz 7]{Gr} (see also \cite[Theorem 30.4]{F} or \cite[Theorem 5.3.1]{Fc}), we have $H^1\big(\widetilde X,\Cal O^{\mathrm{GL}(n,\C)})=0$. In particular,
\begin{equation}\label{25.6.16m'''}\widetilde f_{ij}^{}=\widetilde h_i^{}\widetilde h_j^{-1}\quad\text{on}\quad \widetilde U_i\cap \widetilde U_j
\end{equation}for some family $\widetilde h_i\in\Cal O^{\mathrm{GL}(n,\C)}(U_i)$.
Since $\widetilde f_{ij}^{}\widetilde A=\widetilde A\widetilde f_{ij}^{}$, it follows that
$\widetilde  h_i^{-1}\widetilde A\widetilde  h_i^{}=\widetilde h_j^{-1}\widetilde A h_j^{}$ on $\widetilde U_i\cap\widetilde  U_j$.
Hence, there is a well-defined global holomorphic map $\widetilde  B:\widetilde X\to \mathrm{Mat}(n\times n,\C)$ with
\begin{equation}\label{27.6.16'}
\widetilde  B=\widetilde h_i^{-1}\widetilde A\widetilde  h_i^{}\quad \text{on}\quad \widetilde U_i.
\end{equation}Then, by definition of $\widetilde B$,  $\widetilde  B$ and $\widetilde A$ are locally holomorphically similar on $\widetilde X$. Therefore, by Guralnick's result (Theorem \ref{24.10.16+} above), we can find a holomorphic $\widetilde  H:\widetilde X\to \mathrm{GL}(n,\C)$ with
\begin{equation}\label{27.6.16''}
\widetilde H\widetilde  B\widetilde H^{-1}=\widetilde A\quad\text{on}\quad \widetilde X.
\end{equation}Then, by \eqref{27.6.16'} and \eqref{27.6.16''}, $\widetilde H\widetilde h_{i}^{-1}\widetilde A \widetilde h_i^{}\widetilde H^{-1}_{}=\widetilde H\widetilde B\widetilde A=\widetilde A$ on $\widetilde U_i$, i.e.,
\begin{equation}\label{27.6.16'''}
\widetilde H\widetilde h_{i}^{-1}\in \Cal O^{\mathrm{GCom\,}\widetilde A}(\widetilde U_i).
\end{equation}
By \eqref{25.6.16m''}, for each $\xi\in S$, $\pi^{-1}(V_\xi)\subseteq \widetilde U_{\tau(\xi)}$. Since $\pi^{-1}(V_\xi)$ is a neighborhood of the finite set $\pi^{-1}(\xi)$ and since, by \eqref{27.6.16'''}, $\widetilde H \widetilde h^{-1}_{\tau(\xi)}\in \Cal O^{\mathrm{GCom\,}\widetilde A}\big(\widetilde U_{\tau(\xi)}\big)$ if $\xi\in S$, this implies by Lemma \ref{25.6.16}, that, for each $\xi\in S$, there exist neighborhoods $ \widetilde W_1\big(\pi^{-1}(\xi)\big)$ and $\widetilde W_2\big(\pi^{-1}(\xi)\big)$ of $\pi^{-1}(\xi)$ and a map
\begin{equation}\label{22.10.16''}\widetilde C_{\xi}\in \widehat{\Cal C}^{\mathrm{GCom\,}\widetilde A}(\widetilde X)
\end{equation} such that
\begin{align}&\label{22.10.16}\widetilde W_1\big(\pi^{-1}(\xi)\big)\subseteq \pi^{-1}(V_\xi),\\
&\label{26.6.16}\overline{\widetilde W_2\big(\pi^{-1}(\xi)\big)}\subseteq \widetilde W_1\big(\pi^{-1}(\xi)\big),\\
&\label{26.6.16'}\widetilde C_\xi=\widetilde H\widetilde h_{\tau(\xi)}^{-1}\quad \text{on}\quad  \widetilde W_2\big(\pi^{-1}(\xi)\big),\\
&\label{26.6.16''}\widetilde C_\xi=I\quad \text{on}\quad \widetilde X\setminus \widetilde W_1\big(\pi^{-1}(\xi)\big).
\end{align}
Set
\[
 \widetilde W_1=\bigcup_{\xi\in S}\widetilde W_1\big(\pi^{-1}(\xi)\big)\quad\text{and}\quad \widetilde W_2=\bigcup_{\xi\in S}  \widetilde W_2\big(\pi^{-1}(\xi)\big).
\]By  \eqref{26.6.16+} and \eqref{25.6.16m''}, $V_\xi\cap V_\eta=\emptyset$ and, hence,
$\pi^{-1}(V_\xi)\cap \pi^{-1}(V_\eta)=\emptyset$ if $\xi,\eta\in S$ with $\xi\not=\eta$. By \eqref{22.10.16} this implies that $\widetilde W_1\big(\pi^{-1}(\xi)\big)\cap \widetilde W_1\big(\pi^{-1}(\eta)\big)=\emptyset$ if $\xi,\eta\in S$ with $\xi\not=\eta$.
Therefore and by \eqref{22.10.16''} and \eqref{26.6.16''}, there is a well-defined map
\begin{equation}\label{22.10.16'''}\widetilde C\in \widehat{\Cal C}^{\mathrm{GCom\,}\widetilde A}
(\widetilde X)
\end{equation} with
\begin{align}&\label{23.10.16}\widetilde C =\widetilde C_\xi\quad\text{on}\quad \widetilde W_2\big(\pi^{-1}(\xi)\big),\quad\text{for each}\quad\xi\in S, \quad\text{and}\\
&\label{23.10.16'}\widetilde C=I\quad\text{on}\quad\text{on}\quad \widetilde X\setminus \widetilde W_1.
\end{align}
Define $\widetilde c_i=\widetilde h_i\widetilde H^{-1}\widetilde C$ on $\widetilde U_i$.
  Then, by \eqref{22.10.16'''} and \eqref{27.6.16'''},
 \begin{equation}\label{22.10.16--}
\widetilde c_i\in \widehat{\Cal C}^{\mathrm{GCom\,}\widetilde A}(\widetilde U_i)\quad\text{for all}\quad i\in I,
\end{equation} and, by \eqref{25.6.16m'''},
\begin{equation}\label{22.10.16-}
\widetilde f_{ij}^{}=\widetilde c_i^{}\widetilde c_j^{-1}\quad \text{on}\quad \widetilde U_i\cap \widetilde U_j.
\end{equation}

Now it remains to find maps $c_i\in \widehat{\Cal C}^{\mathrm{GCom\,}A}(U_i)$, $i\in I$, with
\begin{equation}\label{23.10.16+}
\widetilde c_i=c_i\circ \pi\quad\text{on}\quad U_i.
\end{equation}Indeed,  since $\pi$ is biholomorphic from $\widetilde X\setminus \widetilde S$ onto $X\setminus S$, then it follows from \eqref{23.10.16+} and \eqref{22.10.16-} that $f_{ij}^{}=c_i^{}c_j^{-1}$ on $U_i\cap U_j$, i.e., $\{f_{ij}\}$ is $\widehat{\Cal C}^{\mathrm{GCom\,}A}$-trivial, which implies by Proposition \ref{28.6.16-} that $\{f_{ij}\}$ is $\Cal O^{\mathrm{GCom\,}A}$-trivial.

If $i\in I\setminus\tau(S)$, then $U_i\subseteq X\setminus S$, by \eqref{26.6.16+}. Therefore, since $\pi$ is biholomorphic from $\widetilde X\setminus\widetilde S$ onto $X\setminus S$, then we can define $c_i=\widetilde c_i\circ\pi^{-1}$.

Let $\xi\in S$. Denote by $X_\xi$ the set of  germs of $X$ at $\xi$. By
Puiseux's theorem (see, e.g., \cite[Ch. VI, \S 4.1]{L}), for each $\widetilde \xi\in\pi^{-1}(\xi)$, $\pi$ is homeomorphic from some neighborhood of $\widetilde \xi$ onto a representative of one of the germs from $X_\xi$. This implies that there is a neighborhood  of $\xi$ in $X$, $W_2(\xi)$, with $\pi^{-1}\big(W_2(\xi)\big)\subseteq \widetilde W_2\big(\pi^{-1}(\xi)\big)$. Therefore,  it
 follows from  \eqref{23.10.16} and \eqref{26.6.16'} that
\begin{equation}\label{23.10.16-}
\widetilde c_{\tau(\xi)}^{}=\widetilde h_{\tau(\xi)} \widetilde H_{}^{-1}\widetilde C=\widetilde h_{\tau(\xi)} \widetilde H_{}^{-1}\widetilde C_\xi^{}=I\quad\text{on}\quad \pi^{-1}\big(W_2(\xi)\big)\quad\text{for all}\quad\xi\in S.
\end{equation}
By  \eqref{26.6.16+}, $S\cap U_{\tau(\xi)}=\{\xi\}$. Therefore, $\pi$ is biholomorphic from $\widetilde U_{\tau(\xi)}\setminus \pi^{-1}(\xi)$ onto $U_{\tau(\xi)}\setminus \{\xi\}$. By \eqref{22.10.16--} and \eqref{23.10.16-} this implies that there is a  well-defined map $c_{\tau(\xi)}^{}\in\widehat{\Cal C}^{\mathrm{GCom\,} A} (U_{\tau(\xi)})$ with $\widetilde c_{\tau(\xi)}^{}= c_{\tau(\xi)}^{}\circ \pi$ on $\widetilde U_{\tau(\xi)}$, i.e., we have \eqref{23.10.16+} for $i=\tau(\xi)$.
\end{proof}

\begin{rem}\label{3.9.16**} Lemma \ref{17.10.16--} contains the statement $H^1(X,\Cal O^{\mathrm{GCom\,}\Phi})=0$, for each matrix $\Phi\in \mathrm{Mat}(n\times n,\C)$ and each one-dimensional Stein space $X$.  Since
$\mathrm{GCom\,}(\Phi)$ is connected (Lemma \ref{19.1.16''}), this is a special case of the statement
\begin{equation}\label{27.8.16'}
H^1(X,\Cal O^G)=0,
\end{equation}for each connected complex Lie group $G$ and each one-dimensional Stein space $X$.
If $X$ is smooth, \eqref{27.8.16'} was proved by H. Grauert \cite[Satz 7]{Gr}.

For non-smooth $X$, surprisingly, it seems that there is no explicit  reference for \eqref{27.8.16'} in the literature, except for $G=\mathrm{GL}(n,\C)$, see \cite[Theorem 7.3.1 (c) or Corollary 7.3.2, 1.)]{Fc}. Therefore I asked colleagues and got two answers.

F. Forstneri\v c answered  that, by \cite{H}, each one-dimensional Stein space has the homotopy type of a one-dimensional CW complex and, therefore,
\begin{equation}\label{27.8.16''}H^1(X,\Cal C^G)=0, \quad\footnote{Indeed, let $f$ be a $\Cal C_X^G$ cocycle, and let $B$ be the principal $G$-bundle defined by $f$. Then (by definition of $B$) the $\Cal C_X^G$-triviality of $f$ (which we have to prove) is equivalent to the existence of a global continuous section of $B$, and the existence of such a global continuous section follows, e.g., from  \cite[Theorem 11.5 and \S 29.1]{St}.}
\end{equation} which then implies \eqref{27.8.16'} by Grauert's Oka priciple \cite[Satz I]{Gr} (see also \cite[7.2.1]{Fc}).

J. Ruppenthal proposed to pass to the normalization of $X$,  which is smooth. At least if $X$ is locally and globally irreducible and, hence, homeomorphic to its normalization, this immediately reduces the topological statement \eqref{27.8.16''} to the smooth case, which then  implies \eqref{27.8.16'}, again by Grauert's Oka principle. This idea is used  in the proof of Lemma \ref{17.10.16--} above.
\end{rem}

{\em Proof of Theorem \ref{18.10.16}.}
Since $A$ and $B$ are locally holomorphically similar at each point of $X$,  we can find an open covering $\{U_i\}_{i\in I}$ of $X$ and holomorphic  maps $H_i:U_i\to\mathrm{GL}(n,\C)$, $i\in I$, such that
\begin{equation}\label{28.6.16}B=H_i^{-1}AH_i^{}\quad\text{on}\quad U_i.
\end{equation} It follows that $AH_i^{}H_j^{-1}=H_i^{}H_j^{-1}A$ on $U_i\cap U_j$. Hence, $\{H_i^{}H_j^{-1}\}^{}_{i,j\in I}$ is an  $\Cal O^{\mathrm{GCom\,}A}$-cocycle. By Lemma \ref{17.10.16--}, this cocycle is $\Cal O^{\mathrm{GCom\,}A}$-trivial, i.e., $H_i^{}H_j^{-1}=h_i^{}h_j^{-1}$ on $U_i\cap U_j$, for some family $h_i\in\Cal O^{\mathrm{GCom\,}A}(U_i)$. Hence $h_i^{-1}H_i^{}=h_j^{-1}H_j^{}$ on $U_i\cap U_j$, which means that  there is a well-defined  map $H\in\Cal O^{\mathrm{GCom\,}A}(X)$ with $H=i_i^{-1}H_i^{}$ on $U_i$.
 From  \eqref{28.6.16} and the relations $H_i^{}AH_i^{-1}=A$ it follows  that $B=H^{-1}AH$ on $X$.
 \qed

\section{Proof of Theorem \ref{28.6.16**}{}}\label{local}

We show that the statements of this theorem are known or easily  follow from known results.
First we collect these known results.

We begin with following deep result of K. Spallek, which is a special case of  \cite[Satz 5.4]{Sp1} (see also the beginning of \cite{Sp2}).

\begin{prop}\label{8.12.15'}Let $X$ be a complex space, $M:X\to \mathrm{Mat}(n\times m,\C)$  holomorphic, and $\xi\in X$. Then there exists
$k\in \N$ (depending on $M$ and $\xi$) such that the following holds.

Suppose $U$ is a neighborhood of $\xi$ and $f:U\to \C^m$ is  a $\Cal C^k$ map with
$Mf=0$ on $U$.
Then there exist a neighborhood $V\subset U$ of $\xi$ and a holomorpic map  $h:V\to \C^m$ with $Mh=0$ on $V$ and $h(\xi)=f(\xi)$.
\end{prop}

The next proposition is well-known and more easy to prove.

\begin{prop}\label{6.10.16-} Let $D$ be a domain in $\C$,  $M:D\to \mathrm{Mat}(n\times m,\C)$  holomorphic, $M\not\equiv 0$, and $\xi\in D$. Then:

{\em (i)} There exist an open neighborhood $U$  of $\xi$, holomorphic maps $E:U\to \mathrm{GL}(n,\C)$, $F:U\to \mathrm{GL}(m,\C)$, and nonnegative  integers $\kappa_{1},\ldots,\kappa_{r}$
such that
\begin{equation}\label{6.10.16*}
M(\zeta)=E(\zeta)\begin{pmatrix}\Delta(\zeta)&0\\0&0\end{pmatrix}F(\zeta)\quad\text{for all}\quad \zeta\in U,\qquad\footnote{Possibly, some of the zeros in this block matrix have to be omitted.}
\end{equation}where $\Delta(\zeta)$ is the diagonal matrix with the diagonal $(\zeta-\xi)^{\kappa_{1}},\ldots,(\zeta-\xi)^{\kappa_{r}}$.

{\em (ii)} Let $W\subseteq D$ be a neighborhood of $\xi$ and  $c:W\to \C^m$ a continuous map with
\begin{equation}\label{6.10.16++}
Mc=0\quad\text{on}\quad W.
\end{equation}Then there exist a neighborhood $V\subseteq W$ of $\xi$ and a holomorphic map $h:V\to \C^m$ with
\begin{equation*}
Mh=0\quad\text{on}\quad V\qquad\text{and}\qquad h(\xi)=c(\xi).
\end{equation*}
\end{prop}
Part (i) is an application of  the Smith factorization theorem (see, e.g., \cite[Ch. III, Sect. 8]{J})
 to the ring of germs of holomorphic functions in neighborhoods of $\xi$. (A direct proof of part (i) can be found, e.g., in  \cite[Theorem 4.3.1]{GL}).

 Part (ii) is a corollary of part (i). Indeed, let $U$, $E$, $F$ and $r$ be  as in part (i), and let $W$ and $c$ be as in part (ii). Set
 \[
 \begin{pmatrix}f_1(\zeta)\\\vdots\\f_m(\zeta)\end{pmatrix}=F(\zeta)c(\zeta)\quad\text{for}\quad \zeta\in U\cap W.
 \]Then, by  \eqref{6.10.16*} and \eqref{6.10.16++}, $f_{1}(\zeta)=\ldots=f_r(\zeta)=0$ for $\zeta\in (U\cap W)\setminus\{\xi\}$, and, hence, by continuity,  $f_{1}(\xi)=\ldots=f_r(\xi)=0$. It remains to define
 \[h(\zeta)=F(\zeta)^{-1}\begin{pmatrix}0\\\vdots\\0\\f_{r+1}(\xi)\\\vdots\\f_m(\xi)\end{pmatrix}\quad \text{for}\quad \zeta\in V:=U\cap W.\]

Finally, note the following fact, which is nowadays well-known. Proofs can be found, e.g.,  in \cite{W}\footnote{Lemma \ref{9.12.15} is not explicitly stated in \cite{W},
 but it follows immediately from Lemma 1 of \cite{W}. Also, in \cite{W}, $X$ is a domain in the complex plane,  but the proof given there works
also in the general case.} or in \cite[Corollary 2]{Sh}.
\begin{prop}\label{9.12.15} Let $X$ be a complex space, $M:X\to \mathrm{Mat}(n\times m,\C)$ holomorphic, and $\xi\in X$ such that the dimension of $\ke M(\zeta)$ does not depend on $\zeta$ in some neigborhood of $\xi$.   Then
 there exist a neighborhood $U$ of $\xi$ such that the family $\{\ke M(\zeta)\}_{\zeta\in U}$ is a holomorphic sub-vector bundle of $U\times \C^m$.
\end{prop}

{\em Proof of Theorem \ref{28.6.16**}.}   Denote by $\mathrm{End} \big(\mathrm{Mat}(n\times n,\C)\big)$  the space of linear endomorphisms of $\mathrm{Mat}(n\times n,\C)$, and let $\varphi^{}_{A,B}:X\to \mathrm{End} (\mathrm{Mat}(n\times n,\C))$  be  the holomorphic map
defined by
\[
\varphi_{A,B}(\zeta)\Phi= A(\zeta)\Phi-\Phi B(\zeta)\quad\text{for}\quad \zeta\in X\text{ and }\Phi\in \mathrm{Mat}(n\times n,\C).
\]Fix a basis of $\mathrm{Mat}(n\times n,\C)$, and let $M_{A,B}^{}$ be the representation matrix of $\varphi^{}_{A,B}$ with respect to this basis.

First assume that condition (i) is satisfied. Then the claim of the theorem was proved by W. Wasow \cite{W}. He considered only the case when $X$ is a domain in $\C$, but his proof works also in the general case. It goes as follows:

By definition of $\varphi^{}_{A,B}$, \eqref{30.6.16} is the kernel of $\varphi^{}_{A,B}(\zeta)$. Therefore, we have
a neighborhood $U$ of $\xi$ and a number $r\in\N$ such that
\[
\dim\ke\varphi(\zeta)=r\quad\text{for all}\quad \zeta\in U.
\] By Proposition \ref{9.12.15}, this means that the family $\{\ke \varphi^{}_{A,B}(\zeta)\}_{\zeta\in U}$ is a holomorphic sub-vector bundle of the product bundle $U\times \mathrm{Mat}(n\times n,\C)$. Since $\Phi\in \ke \varphi^{}_{A,B}(\xi)$, then,  after shrinking $U$, we we can find a holomorphic section $H$ of this bundle with $H(\xi)=\Phi$. \qed

If (ii) or (iii) is satisfied, then the claim of the theorem follows immediately from
Propositions  \ref{6.10.16-} (ii) and  \ref{8.12.15'}, respectively, with $M=M^{}_{A,B}$. \qed

\begin{rem}\label{3.9.16*} This proof shows that condition (iii) in Theorem \ref{28.6.16**} can be replaced by the following:
There exists a  positive integer $k$ depending on $\xi$, $A$ and $B$ such that, if there exist a neighborhood $U$ of $\xi$ and a $\Cal C^k$ map $T:V\to \mathrm{Mat}(n\times n,\C)$ such that $T(\xi)=\Phi$ and $TB=AT$ on $U$,
then  there exist a neighborhood $V\subseteq U$ of $\xi$ and a holomorphic map $H:V\to \mathrm{Mat}(n\times n,\C)$ with  $H(\xi)=\Phi$ and
$HB=AH$ on $V$.
\end{rem}

\section{Local counterexamples}\label{local examples}

Let $z$ and $w$ be the canonical complex coordinate functions on $\C^2$.

We begin with the following observation of O. Forster and K. J. Ramspott  \cite[page 159]{FR2}): If $\alpha$, $\beta$ and $\gamma$ are holomorphic functions in a neighborhood of the origin in $\C^2$, which solve the equation
\[
\alpha z^3+\beta w^3+\gamma z^2w^2=0
\]in this neighborhood, then, comparing the coefficients in the Taylor series, it follows easily that $\alpha(0)=\beta(0)=\gamma(0)=0$.
With continuous functions however, this equation can be solved  with $\gamma(0)\not=0$. For example,
\[
\frac{\overline z w^{2}}{\vert z\vert^{2}+\vert w\vert^2}z^3+\frac{\overline w z^{2}}{\vert z\vert^2+\vert w\vert^2}w^3=z^2w^2.
\] We use a $\Cal C^\ell$-version of this.
\begin{leer}\label{2.9.16'''} Let
$\B^2$ be the open unit ball in $\C^2$, $\ell\in \N$,
\begin{align*}&A=\begin{pmatrix} z^{2+\ell}w^{2+\ell}&z^{3+\ell}\\w^{3+\ell}&0\end{pmatrix},\qquad B=\begin{pmatrix} 0&z^{3+\ell}\\w^{3+\ell}&z^{2+\ell}w^{2+\ell}\end{pmatrix},\\
&c^{}_z=\frac{\overline z w^{2+\ell}}{\vert z\vert^{2}+\vert w\vert^2},\quad
c_w^{}=\frac{\overline w z^{2+\ell}}{\vert z\vert^2+\vert w\vert^2},\quad S=\begin{pmatrix} 1&c_w\\-c_z&1\end{pmatrix}.
\end{align*}
Then it is again easy to see  that
\begin{equation}\label{5.1.16-}
c^{}_zz^{3+\ell}+c^{}_ww^{3+\ell}=z^{2+\ell}w^{2+\ell}\quad\text{on}\quad\B^2,
\end{equation}and, comparing the coefficients of the Taylor series\footnote{Below we explain this in detail in the case of Lemmas \ref{6.1.16} and \ref{6.1.16'}, each of which  is stronger than Lemma \ref{5.1.16}.}, we get
\end{leer}
\begin{lem}\label{5.1.16}
Suppose $\alpha$, $\beta$, $\gamma$ are holomorphic functions in a neighborhood  of the origin in $\C^2$ such that
\begin{equation*}
\alpha z^{\ell+3}+\beta w^{\ell+3}+\gamma z^{\ell+2}w^{\ell+2}=0
\end{equation*} in this neighborhood. Then $\alpha(0)=\beta(0)=\gamma(0)=0$.
\end{lem}
Also it is easy to see that the functions
 $c_z$ and $c_w$ are of class $\Cal C^\ell$ on $\C^2$ and that $\vert c_zc_w\vert<1$ on $\B^2$. Hence
$S$ is of class  $\Cal C^\ell$  on $\C^2$, and $S(\zeta)\in \mathrm{GL}(n,\C)$ for all $\zeta\in\B^2$. Moreover,
\begin{align*}
&AS=\begin{pmatrix} z^{2+\ell}w^{2+\ell}-c_zz^{3+\ell}&c_wz^{\ell+2}w^{\ell+2}+z^{3+\ell}\\w^{3+\ell}&c_w^{}w^{3+\ell}\end{pmatrix},\\
&SB=\begin{pmatrix} c_w^{}w^{3+\ell}&z^{3+\ell}+c_wz^{2+\ell}w^{2+\ell}\\w^{3+\ell}&-c_zz^{3+\ell}+z^{2+\ell}w^{2+\ell}\end{pmatrix},
\end{align*} which implies by \eqref{5.1.16-} that
\begin{equation}\label{2.9.16''}SBS^{-1}=A\quad\text{on}\quad \B^2.
\end{equation}Hence $A$ and $B$ are globally $\Cal C^\ell$ similar on $\B^2$. On the other hand, we have
\begin{lem}\label{5.1.16*} Let $U$ be an open neighborhood of the origin in $\C^2$, and $H:U\to \mathrm{Mat}(2\times 2,\C)$ holomorphic. Then:

{\em (i)} If, on $U$, $AH=HB$  or $HA=BH$, then $H(0)=0$.

{\em (ii)} If, on $U$, $AH=HA$ or $AB=BA$, then  $H(0)=\lambda I_2$ for some $\lambda\in \C$.
\end{lem}

\vspace{2mm}
\noindent {\em Proof.} Let $H=\big(\begin{smallmatrix} a&b\\c&d\end{smallmatrix}\big)$. Then
\begin{align}
&\label{3.9.16}AH=\begin{pmatrix} a z^{2+\ell}w^{2+\ell}+c z^{3+\ell}&b z^{2+\ell} w^{2+\ell}+d z^{3+\ell}\\a w^{3+\ell}&b w^{3+\ell}\end{pmatrix},\\
&\label{3.9.16'}HB=\begin{pmatrix} b w^{3+\ell}&a z^{3+\ell}+b z^{2+\ell}w^{2+\ell}\\d w^{3+\ell}&c z^{3+\ell}+dz^{2+\ell} w^{2+\ell}\end{pmatrix},\\
&\label{3.9.16''}HA=\begin{pmatrix} a z^{2+\ell}w^{2+\ell}+b w^{3+\ell}&a z^{3+\ell}\\c z^{2+\ell}w^{2+\ell}+dw^{3+\ell}&cz^{3+\ell}\end{pmatrix},\\
&\label{3.9.16'''}BH=\begin{pmatrix} c z^{3+\ell}&d z^{3+\ell}\\aw^{3+\ell}+cz^{2+\ell}w^{2+\ell}&bw^{3+\ell}+dz^{2+\ell}w^{2+\ell}\end{pmatrix}.
\end{align}  In particular:
\begin{align*}
&\text{if }AH=HB,\text{ then }az^{2+\ell}w^{2+\ell}+c z^{3+\ell}= b w^{3+\ell}=c z^{3+\ell}+d z^{2+\ell} w^{2+\ell},\\
& \text{if }HA=BH,\text{ then }a z^{2+\ell}w^{2+\ell}+b w^{3+\ell}=cz^{3+\ell}=bw^{3+\ell}+dz^{2+\ell}w^{2+\ell},\\
&\text{if }AH=HA, \text{ then }cz^{3+\ell}=bw^{3+\ell} \text{ and }(a-d)z^{3+\ell}=bz^{2+\ell}w^{2+\ell},\\
&\text{if }BH=HB, \text{ then }bw^{3+\ell}=cz^{3+\ell} \text{ and }(d-a)w^{3+\ell}=cz^{2+\ell}w^{2+\ell}.
\end{align*} By Lemma \ref{5.1.16}, this yields:
\begin{align*}
&\text{if }AH=HB\text{ or }HA=BH,\text{ then } a(0)=b(0)=c(0)=d(0)=0,\\
&\text{if }AH=HA \text{ or }BH=HB, \text{ then }b(0)=c(0)=0\text{ and }a(0)=d(0).\quad\qed
\end{align*}

Lemma  \ref{5.1.16*} (i) in particular says that $A$ and $B$ are not locally holomorphically similar at $0$. At the end of this section we prove the following stronger
\begin{thm}\label{6.1.16''} Suppose {\em (a)} $X=\{z^p=w^q\}$, where $p,q\in\N$ such that $\ell+2< q<p$ and $p$, $q$ are  relatively prime,
or {\em (b)} $X$ is the union of  $2\ell+5$ pairwise different one-dimensional linear
subspaces of $\C^2$.

Then the restrictions $A\big\vert_X$ and $B\big\vert_X$ are not locally holomorphically similar at $0$.
\end{thm}

\begin{lem}\label{6.1.16} Let $X=\{z^p=w^q\}$, where $p,q\in\N$ such that $\ell+2< q<p$ and $p$, $q$ are  relatively prime.
Suppose $U$ is a neighborhood of the origin in $\C^2$, and  $\alpha,\beta,\gamma:U\to \C$ are holomorphic such that
\begin{equation}\label{29.11.15'neu}
\alpha z^{\ell+3}+\beta w^{\ell+3}+\gamma z^{\ell+2}w^{\ell+2}=0\quad\text{on}\quad X\cap U.
\end{equation} Then $\alpha(0)=\beta(0)=\gamma(0)=0$.
\end{lem}

\begin{proof} Choose $0<\varepsilon<1$ so small that the closed bidisk $\max(\vert z\vert,\vert w\vert)\le \varepsilon$ is contained in $U$, and let
\[
\sum_{j,k=0}^\infty \alpha_{jk}z^j w^k,\quad \sum_{j,k=0}^\infty \beta_{jk}z^j w^k,\quad\sum_{j,k=0}^\infty \gamma_{jk}z^j w^k
\] be the Taylor series  of $\alpha$, $\beta$ and $\gamma$, respectively.  Then, by \eqref{29.11.15'neu},
\[
\sum_{j,k=0}^\infty \alpha_{jk}z^{j+\ell+3} w^k+\sum_{j,k=0}^\infty \beta_{jk}z^j w^{k+\ell+3}+\sum_{j,k=0}^\infty \gamma_{jk}z^{j+\ell+2} w^{k+\ell+2}=0
\]if $z^p=w^q$ and $\max(\vert z\vert,\vert w\vert)< \varepsilon$. With $z=t^q$ and $w=t^p$ for $0\le t< \varepsilon$, this yields
\begin{equation*}
\sum_{j,k=0}^\infty \alpha_{jk}t_{}^{(j+\ell+3)q+kp} + \sum_{j,k=0}^\infty \beta_{jk}t_{}^{jq+(k+\ell+3)p}+\sum_{j,k=0}^\infty \gamma_{jk}t_{}^{(j+\ell+2)q+(k+\ell+2)p}=0
\end{equation*}for all  $0\le t< \varepsilon$. Comparing the coefficients of $t^{(\ell+3)q}$, $t^{(\ell+3)p}$ and $t^{(\ell+2)(p+q)}$, we get
\begin{align*}
\alpha^{}_{00}+\sum_{(j,k)\in A_\beta} \beta_{jk}+\sum_{(j,k)\in A_\gamma}\gamma_{jk}&=0,\\
\sum_{(j,k)\in B_\alpha} \alpha_{jk}+\beta^{}_{00}+\sum_{(j,k)\in B_\gamma}\gamma^{}_{jk}&=0,\\
\sum_{(j,k)\in C_\alpha} \alpha_{jk}+\sum_{(j,k)\in C_\beta}\beta^{}_{jk}+\gamma^{}_{00}&=0,
\end{align*}
where $A_\beta,\ldots, C_\beta$  are the subsets of $\N\times \N$ defined by
\begin{align*}
(j,k)\in A_\beta&\overset{\mathrm{def}}{\Longleftrightarrow}jq+(k+\ell+3)p=(\ell+3)q\Longleftrightarrow(k+\ell+3)p=(\ell+3-j)q,\\
(j,k)\in A_\gamma&\overset{\mathrm{def}}{\Longleftrightarrow}(j+\ell+2)q+(k+\ell+2)p=(\ell+3)q\Longleftrightarrow (k+\ell+2)p=(1-j)q,\\
(j,k)\in B_\alpha&\overset{\mathrm{def}}{\Longleftrightarrow}(j+\ell+3)q+kp=(\ell+3)p\Longleftrightarrow (j+\ell+3)q=(\ell+3-k)p,\\
 (j,k)\in B_\gamma&\overset{\mathrm{def}}{\Longleftrightarrow}(j+\ell+2)q+(k+\ell+2)p =(\ell+3)p\Longleftrightarrow (j+\ell+2)q=(1-k)p,\\
(j,k)\in C_\alpha&\overset{\mathrm{def}}{\Longleftrightarrow}(j+\ell+3)q+kp=(\ell+2)(p+q)\Longleftrightarrow (j+1)q=(\ell+2-k)p,\\
(j,k)\in C_\beta&\overset{\mathrm{def}}{\Longleftrightarrow}jq+(k+\ell+3)p=(\ell+2)(p+q)\Longleftrightarrow (\ell+2-j)q=(k+1)p.
\end{align*}
It is sufficient to prove that $A_\beta=A_\gamma=B_\alpha=B_\gamma=C_\alpha=C_\beta=\emptyset$.

Assume $(k+\ell+3)p=(\ell+3-j)q$. Contrary to  $q<p$, then it follows
\[
p=\frac{\ell+3-j}{k+\ell+3}q\le \frac{\ell+3}{k+\ell+3}\le q.
\]

Assume  $(k+\ell+2)p=(1-j)q$. Contrary to $p>p/2$, then it follows
\[
p=\frac{1-j}{k+\ell+2}q\le \frac{q}{2}<\frac{p}{2}.
\]

Assume $(j+\ell+3)q=(\ell+3-k)p$. Since $p$ and $q$ are relatively prime, this implies that $j+\ell+3=np$, for some integer $n\in \N^*$. $n=1$ is not possible, for this would imply that $p=j+\ell+3\le \ell+3\le q<p$.  $n\ge 2$ is also impossible, as this would imply that
$p\ge \ell+3\ge j+\ell+3\ge 2p$.

Assume $(j+\ell+2)q=(1-k)p$. This implies that $k=0$ and therefore $(j+\ell+2)q=p$, which is not possible, since $p$ and $q$ are relatively prime.

Assume $(j+1)q=(\ell+2-k)p$.  As $p$ and $q$ are relatively prime, this implies that $\ell +2-k$ is positive and can be divided by $q$. In particular, $\ell+2-k\ge q$, which is not possible, for $\ell+2-k<q$.

Assume $(\ell+2-j)q=(k+1)p$. Since $p$ and $q$ are relatively prime, this implies that $\ell+2-j=np$ for some $n\in\N^*$ and, further,  $p>\ell+2\ge\ell+2-j=np\ge p$, contrary to $q<p$.
\end{proof}

\begin{lem}\label{6.1.16'}Let $t_1,\ldots,t_{2\ell+5}$ be pairwise different complex numbers, and
\[
X:=\bigcup_{j=1}^{2\ell+5}\{w=t_j z\}.
\]
Suppose $U$ is a neighborhood of the origin in $\C^2$, and  $\alpha,\beta,\gamma:U\to \C$ are holomorphic such that
\begin{equation}\label{29.11.15'}
\alpha z^{\ell+3}+\beta w^{\ell+3}+\gamma z^{\ell+2}w^{\ell+2}=0\quad\text{on}\quad X\cap U.
\end{equation} Then $\alpha(0)=\beta(0)=\gamma(0)=0$.
\end{lem}

\begin{proof}

To prove that $\alpha(0)=0$, we assume that $\alpha(0)\not=0$. Setting $b=\beta/\alpha$ and $c=\gamma/\alpha$, then we get holomorphic functions $b, c$ in a neighorhood $V\subseteq U$ of $0$ such that
\[
z^{\ell+3}=c z^{2+\ell}w^{\ell+2}-b w^{\ell+3}=0\quad \text{on}\quad X\cap V.
\]It follows that, for $1\le j\le 2\ell+ 5$ and all $\zeta$ in some neighborhood  of zero in the complex plane,
\[\zeta^{\ell+3}=c(\zeta,t_j\zeta) \zeta^{2\ell+4}t_j^{\ell+2}-b(\zeta,t_j\zeta)\zeta^{\ell+3}t_j^{\ell+3}
\]and, hence,
\[1=c(\zeta,t_j\zeta) \zeta^{\ell+1}t_j^{\ell+2}-b(\zeta,t_j\zeta)t_j^{\ell+3}.
\]Hence, $1=-t_j^{\ell+3}\beta(0,0)$ for $1\le j\le 2\ell+5$. This implies that $\beta(0,0)\not=0$ and $t_1,\ldots,t_{2\ell+5}$ are
solutions of the equation
\[
t^{\ell+3}=-\frac{1}{\beta(0,0)}.
\]As $2\ell+5>\ell+3$ and the numbers $t_j$ are pairwise different, this is impossible.

Changing the roles of $z$ and $w$, one proves in the same way  that $\beta(0)=0$.

Finally we assume that $\gamma(0)\not=0$. Setting $a=\alpha/\gamma$ and $b=\beta/\gamma$, then we
get holomorphic functions $a,b$ in a neighborhood $V\subseteq U$ of $0$ such that
\[
a z^{\ell+3}+b w^{\ell+3}=z^{\ell+2}w^{\ell+2}\quad\text{on}\quad X\cap V.
\]It follows that, for $1\le j\le 2\ell+ 5$ and all $\zeta$ in some neighborhood $\Omega$ of zero in the complex plane
\[
a(\zeta,t_j\zeta ) \zeta^{\ell+3}+b(\zeta,t_j\zeta) \zeta^{\ell+3}t_j^{\ell+3}=\zeta^{2\ell+4}t_j^{\ell+2}
\]and, hence,
\begin{equation*}\label{8.11.15}
a(\zeta,t_j\zeta) +b(\zeta,t_j\zeta)t_j^{\ell+3} =\zeta^{\ell+1}t_j^{\ell+2}.
\end{equation*} If $\sum a_{\mu\nu}z^\mu w^\nu$ and $\sum b_{\mu\nu}z^\mu w^\nu$ are the Taylor series at the origin
of $a$ and $b$, respectively,
this means that
\[\sum_{\mu,\nu=0}^\infty \Big(a_{\mu\nu}t_j^\nu+b_{\mu\nu} t_j^{\nu+\ell+3}\Big)\zeta^{\mu+\nu}
=\zeta^{\ell+1}t_j^{\ell+2}
\]for all  $1\le j\le 2\ell+5$ and $\zeta\in\Omega$. Comparing the coefficients of $\zeta^{\ell+1}$, this yields
\[
\sum_{\mu+\nu=\ell+1}a_{\mu\nu}t_j^\nu+\sum_{\mu+\nu=\ell+1}b_{\mu\nu}t_j^{\nu+\ell+3}=t_j^{\ell+2},\quad\quad 1\le j\le 2\ell+5.
\]i.e.,
\[
\sum_{\nu=0}^{\ell+1}\alpha^{}_{\ell+1-\nu,\nu}t_j^\nu+\sum_{\nu=\ell+3}^{2\ell+4}\beta^{}_{2\ell+4-\nu,\nu-\ell-3}t_j^{\nu}=t_j^{\ell+2},
\quad\quad 1\le j\le 2\ell+5.
\]
Hence, the system of $2\ell+5$ linear equations in $2\ell+5$ variables
\[
\sum_{\nu=0}^{2\ell+4}t_j^\nu x_\nu=0, \quad\quad 1\le j\le 2\ell+5
\]has a non-trivial solution, namely one with $x_{\ell+2}^{}=-1$. This not possible, as
\[
\det\begin{pmatrix}1&t^{}_1&t_1^2&\ldots &t^{2\ell+4}_1\\&&\ldots&&\\1&t_{2\ell+5}^{}&t_{2\ell+5}^2&\ldots& t_{2\ell+5}^{2\ell+4}\end{pmatrix}
=\prod_{1\le i<j\le 2\ell+5}(t_i-t_j)\not=0.\]
\end{proof}

{\em Proof of Theorem \ref{6.1.16''}.} Assume there exist a neighborhood $U$ of the origin in $\C^2$ and a holomorphic map $H:U\to\mathrm{GL}(2,\C)$ such that $H^{-1}AH=B$ on $X\cap U$. If $H=\big(\begin{smallmatrix} a&b\\c&d\end{smallmatrix}\big)$, then, by \eqref{3.9.16} and \eqref{3.9.16'}, in particular, it follows that
\[
az^{2+\ell}w^{2+\ell}+c z^{3+\ell}= b w^{3+\ell}=c z^{3+\ell}+d z^{2+\ell} w^{2+\ell}\quad\text{on}\quad X\cap U,
\] which implies by Lemmas \ref{6.1.16} and \ref{6.1.16'} that $a(0)=b(0)=c(0)=d(0)=0$, i.e., $H(0)=0$, which contradicts the assumption that $H(0)$ is invertible.
\qed

\section{A global counterexample}\label{global example}

Let $v_1,v_2,v_3$ denote the canonical complex coordinate functions on $\C^3$, and let $x_j=\rea v_j$ and $y_j=\im v_j$. Set
\begin{align*}h=v_1+&iv_2,\quad h^*=v_1-iv_2,\\
\s^2=\big\{y_1=y_2=y_3=0\big\}&\cap \big\{x_1^2+x_2^2+x_3^2=1\big\},\quad \s^1=\s^2\cap \big\{x_3=0\big\}.
\end{align*}
Then $hh^*= x_1^2+x_2^2=1$ on $\s^1$. Therefore we can find a neighborhood $N(\s^2)$ in $\C^3$ of $\s^2$ and  $\varepsilon>0$   such that
\begin{equation}\label{7.1.16}
\big\vert  hh^*-1\big\vert<\frac{1}{2}\quad\text{on}\quad N(\s^2)\cap \{-2\varepsilon<x_3<2\varepsilon\}.
\end{equation}
Set
\[
\rho=\big(x_1^2+x_2^2+x_3^2-1\big)^3+y_1^2+y_2^2+y_3^2.
\]Then $\s^2=\{\rho=0\}$ and, making $\varepsilon$ smaller, we can achieve that
\[
\s^2_\varepsilon=\{\rho<\varepsilon\}\subseteq N(\s^2).
\]Moreover, we can choose $\varepsilon$ so small that $\rho$ is strictly plurisubharmonic in $\s^2_\varepsilon$. Then
$\s^2_\varepsilon$ is Stein. Set \[U_+=\s^2_\varepsilon\cap\{x_3>-\varepsilon\}\quad\text{and}\quad U_-=\s^2_\varepsilon\cap \{x_3<\varepsilon\}.
\]

\begin{lem}\label{8.1.16} {\em(i)} There exist holomorphic  $a_\pm,b_\pm,c_\pm,d_\pm:U_\pm\to \C$ such that
\begin{align}&\label{8.1.16'}\begin{pmatrix}a_\pm(\zeta)&b_\pm(\zeta)\\c_\pm(\zeta)&d_\pm(\zeta)\end{pmatrix}\in \mathrm{GL}(2,\C)\quad\text{for all}\quad \zeta\in U_\pm,\\&\label{8.1.16''} \begin{pmatrix}h&0\\0&h^*\end{pmatrix}
=\begin{pmatrix}a_+&b_+\\c_+&d_+\end{pmatrix}
\begin{pmatrix}a_-&b_-\\c_-&d_-\end{pmatrix}^{-1}\quad\text{on}\quad U_+\cap U_-.
\end{align}
{\em (ii)} There do not exist continuous functions $f_\pm^{}:U_\pm^{}\to \C^*$ such that
\begin{equation}\label{9.1.16}h=\frac{f_+}{f_-}\quad\text{on}\quad U_+\cap U_-.
\end{equation}
\end{lem}

\begin{proof} (i) Since $\s_\varepsilon^2$ is Stein and $\s^2_\varepsilon=U_+\cup U_-$, by Grauert's Oka principle [Satz I]\cite{Gr}, [Theorem 5.3.1 (ii)]\cite{Fc}, it is sufficient to  find a continuous  $C_+:U_+\to \mathrm{GL}(2,\C)$ with
\begin{equation}\label{8.1.16-}
\begin{pmatrix}h&0\\0&h^*\end{pmatrix}=C^{}_+\quad\text{on}\quad U_+\cap U_-,
\end{equation}which can be done as follows: Take a continuous function $\chi:\R\to[0,1]$ such that $\chi(t)=1$ if $t\le \varepsilon$  and $\chi(t)=0$ if $ t\ge 2\varepsilon$, and define
\[
C_+(\zeta)=\begin{pmatrix}\chi\big(x_3(\zeta)\big)h(\zeta)&1-\chi\big(x_3(\zeta)\big)\\\chi\big(x_3(\zeta)\big)-1&\chi\big(x_3(\zeta)\big)h^*(\zeta)
\end{pmatrix}\quad \text{for}\quad\zeta\in U_+.
\]If $\zeta\in U_+\cap U_-$, then $-\varepsilon<x_3(\zeta)<\varepsilon$ and therefore $\chi\big(x_3(\zeta)\big)=1$, which implies \eqref{8.1.16-}.
It remains to prove that $\det C_+(\zeta)\not=0$  for all $\zeta\in U_+$.
If $\zeta\in U_+$ with $x_3(\zeta)<2\varepsilon$, then,  by \eqref{7.1.16}, $\rea\big(h(\zeta)h^*(\zeta)\big)>1/2$, which yields
\[
\rea\det C_+(\zeta)\ge \frac{1}{2}\Big(\chi\big(x_3(\zeta)\big)\Big)^2+\Big(1-\Big(\chi\big(x_3(\zeta)\big)\Big)^2\ge \frac{1}{2}.
\]If  $\zeta\in U_+$ with $x_3(\zeta)\ge 2\varepsilon$, then $\chi\big(x_3(\zeta)\big)=0$ and, hence, $\det C_+(\zeta)=1$.

 (ii)
Assume such functions exist. Then, for $0\le s\le 1$,
 we have  continuous closed curves $\gamma^+_s:[0,2\pi]\to \C^*$ and $\gamma^-_s:[0,2\pi]\to \C^*$, well-defined by
 \begin{align*}
 &\gamma^+_s(t)=f_+^{}\Big(\big(\sqrt{1-s^2}\cos t,\sqrt{1-s^2}\sin t,s\big)\Big),\\
 &\gamma^-_s(t)=f_-^{}\Big(\big(\sqrt{1-s^2}\cos t,\sqrt{1-s^2}\sin t,-s\big)\Big).
 \end{align*}
 Let
 \[
 \mathrm{Ind\,}\gamma_s^\pm:=\frac{1}{2\pi}\int_0^{2\pi}\frac{(\gamma_s^{\pm})'(t)}{\gamma_s^\pm(t)}\,dt
 \] be the winding number of $\gamma_s^\pm$. It is clear that $\mathrm{Ind\,}\gamma_s^\pm$ depends continuously on $s$, and
 it is well known that $\mathrm{Ind\,}\gamma_s^\pm$ is always an integer. Therefore
 \[
 \mathrm{Ind\,}\gamma_{1}^+=\mathrm{Ind\,}\gamma_0^+\quad\text{and}\quad \mathrm{Ind\,}\gamma_{1}^-=\mathrm{Ind\,}(\gamma_0^-).
 \] Since $\gamma^+_1$ and $\gamma_1^-$ are constant, it follows that $\mathrm{Ind\,}(\gamma_0^+)=\mathrm{Ind\,}(\gamma_0^-)=0$ and, hence,
 \begin{equation}\label{9.1.16'}
 \mathrm{Ind\,}\frac{\gamma_0^+}{\gamma_0^-}=0.
 \end{equation}By definition of $h$, $h\big(\cos t,\sin t,0\big)=\cos t+i\sin t=e^{it}$.
 By \eqref{9.1.16} this yields
\[
e^{it}=\frac{f_+\big(\cos t,\sin t,0\big)}{f_-\big(\cos t,\sin t,0\big)}=\frac{\gamma^+_0(t)}{\gamma^-_0(t)},\qquad 0\le t\le 2\pi,
\]which contradicts \eqref{9.1.16'}.
\end{proof}
Now, using also the notations introduced in  Section \ref{2.9.16'''}, we set
\[X=\s^2_\varepsilon\times \B^2,\qquad  X_\pm=U_\pm\times \B^2,\] and, for $(\zeta,\eta)\in X$,
\[\widetilde A(\zeta,\eta)=A(\eta),\quad \widetilde B(\zeta,\eta)=B(\eta),\quad
\widetilde h(\zeta,\eta)=h(\zeta),\quad\widetilde h^*(\zeta,\eta)=h^*(\zeta).\]
Further, let $a_\pm,b_\pm,c_\pm,d_\pm$ be as in Lemma \ref{8.1.16} (i), and define holomorphic maps $\Theta_\pm:X_\pm\to \mathrm{Mat}(4\times 4,\C)$ by the block matrices
\[
\Theta_\pm(\zeta,\eta)=\begin{pmatrix}a_\pm(\zeta)I_2& b_\pm(\zeta)I_2\\ c_\pm(\zeta)I_2&d_\pm(\zeta)I_2\end{pmatrix},\quad (\zeta,\eta)\in X_\pm.
\]Then, by \eqref{8.1.16'} and \eqref{8.1.16''},
$\Theta_\pm(\zeta,\eta)\in \mathrm{GL}(4,\C)$ for all $(\zeta,\eta)\in X_\pm$, and
\begin{equation}\label{10.1.16}
\begin{pmatrix}\widetilde h I_2&0\\0&\widetilde h^*I_2\end{pmatrix}=\Theta^{}_+\Theta_-^{-1}\quad \text{on}\quad X_+\cap X_-.
\end{equation}Since, obviously,
\[
\begin{pmatrix}\widetilde A&0\\0&\widetilde B\end{pmatrix}\begin{pmatrix}\widetilde h I_2&0\\0&\widetilde h^*I_2\end{pmatrix}=\begin{pmatrix}\widetilde h I_2&0\\0&\widetilde h^*I_2\end{pmatrix}\begin{pmatrix}\widetilde A&0\\0&\widetilde B\end{pmatrix}\quad\text{on}\quad X,
\]this implies that
\begin{equation}\label{2.9.16}
\Theta_+^{}\begin{pmatrix}\widetilde A&0\\0&\widetilde B\end{pmatrix}\Theta_+^{-1}=\Theta_-^{}\begin{pmatrix}\widetilde A&0\\0&\widetilde B\end{pmatrix}\Theta_-^{-1}\quad\text{on}\quad X_+\cap X_-.
\end{equation}Let $\Phi:X\to \mathrm{Mat}(4\times 4,\C)$ be defined by the two sides of \eqref{2.9.16}.

\begin{thm}\label{10.1.16'} $\Phi$ and $\big(\begin{smallmatrix}\widetilde A&0\\0&\widetilde B\end{smallmatrix}\big)$ are

{\em (a)} locally holomorphically similar on $X$,

{\em (b)} globally $\Cal C^\ell$ similar on $X$,

{\em(c)} not globally $\Cal C^\infty$ similar on $X$.
\end{thm}
\begin{proof}The local holomorphic similarity is clear by definition of $\Phi$.

To prove (b), let $S$ be as in Section \ref{2.9.16'''} and $\widetilde S(\zeta,\eta):=S(\eta)$ for $(\zeta,\eta)\in X$.
Since  $a_\pm(\zeta)I_2$, $b_\pm(\zeta)I_2$, $c_\pm(\zeta)I_2$ and $d_\pm(\zeta)I_2$ commute with $A(\eta)$, we have
\begin{equation}\label{10.1.16+}
\begin{pmatrix}\widetilde A&0\\0&\widetilde A\end{pmatrix}\Theta_\pm=\Theta_\pm\begin{pmatrix}\widetilde A&0\\0&\widetilde A\end{pmatrix}\quad\text{on}\quad U_\pm.
\end{equation}Moreover, it is clear that $\widetilde S\widetilde h^*I_2= \widetilde h^*I_2 \widetilde S$ and therefore
\[
\begin{pmatrix}I_2&0\\0&\widetilde S\end{pmatrix}\begin{pmatrix}\widetilde hI_2&0\\0&\widetilde h^*I_2\end{pmatrix}=\begin{pmatrix}\widetilde hI_2&0\\0&\widetilde h^*I_2\end{pmatrix}\begin{pmatrix}I_2&0\\0&\widetilde S\end{pmatrix}\quad\text{on}\quad X,
\]which implies by \eqref{10.1.16} that
\[
\Theta_+^{-1}\begin{pmatrix}I_2&0\\0&\widetilde S\end{pmatrix}\Theta_+^{}=\Theta_-^{-1}\begin{pmatrix}I_2&0\\0&\widetilde S\end{pmatrix}\Theta_-^{}
\quad\text{on}\quad X_+\cap X_-
\]and further
\[
\begin{pmatrix}I_2&0\\0&\widetilde S^{-1}\end{pmatrix}\Theta_+^{-1}\begin{pmatrix}I_2&0\\0&\widetilde S\end{pmatrix}\Theta_+^{}=
\begin{pmatrix}I_2&0\\0&\widetilde S^{-1}\end{pmatrix}\Theta_-^{-1}\begin{pmatrix}I_2&0\\0&\widetilde S\end{pmatrix}\Theta_-^{}
\quad\text{on}\quad X_+\cap X_-.
\]
Let $\Psi:X\to \mathrm{GL}(4,\C)$ be the $\Cal C^\ell$ map defined by the two sides of the last equality. Then, by  \eqref{2.9.16''},
\begin{equation*}
\Psi^{-1}\begin{pmatrix}\widetilde A &0\\0&\widetilde B\end{pmatrix}\Psi=\Theta_\pm^{-1}\begin{pmatrix}I_2&0\\0&\widetilde S^{-1}\end{pmatrix}\Theta_\pm^{}\begin{pmatrix}\widetilde A &0\\0&\widetilde A\end{pmatrix}
\Theta_\pm^{-1}\begin{pmatrix}I_2&0\\0&\widetilde S\end{pmatrix}\Theta_\pm^{}\quad\text{on}\quad X_\pm.
\end{equation*}In view of \eqref{10.1.16+}, this implies that
\[\Psi^{-1}\begin{pmatrix}\widetilde A &0\\0&\widetilde B\end{pmatrix}\Psi=\Theta_\pm^{-1}\begin{pmatrix}I_2&0\\0&\widetilde S^{-1}\end{pmatrix}\begin{pmatrix}\widetilde A &0\\0&\widetilde A\end{pmatrix}\begin{pmatrix}I_2&0\\0&\widetilde S\end{pmatrix}\Theta_\pm^{}\quad\text{on}\quad X_\pm,
\] and further, again by \eqref{2.9.16''},
\[\Psi^{-1}\begin{pmatrix}\widetilde A &0\\0&\widetilde B\end{pmatrix}\Psi=\Theta_\pm^{-1}\begin{pmatrix}\widetilde A &0\\0&\widetilde B\end{pmatrix}
\Theta_\pm^{}\quad\text{on}\quad X_\pm.
\]By definition of $\Phi$, this means that $\Psi^{-1}\big(\begin{smallmatrix}\widetilde A&0\\0&\widetilde A\end{smallmatrix}\big)\Psi=\Phi$ on $X$, which completes the proof of (b).

To prove (c), we assume that  $\Phi$ and $\big(\begin{smallmatrix}\widetilde A&0\\0&\widetilde B\end{smallmatrix}\big)$ are globally $\Cal C^\infty$ similar on $X$. Since $X$ is Stein, then, by Theorem \ref{24.8.16+}, $\Phi$ and $\big(\begin{smallmatrix}\widetilde A&0\\0&\widetilde B\end{smallmatrix}\big)$ are even globally holomorphically  similar on $X$, i.e., we have a holomorphic map $\Theta:X\to \mathrm{GL}(4,\C)$ with
\[
\Theta^{-1}\Phi\Theta=\begin{pmatrix}\widetilde A&0\\0&\widetilde B\end{pmatrix}\quad\text{on}\quad X.
\]By definition of $\Phi$ this means that
\[
\Theta^{-1}_{}\Theta_\pm^{}\begin{pmatrix}\widetilde A&0\\0&\widetilde B\end{pmatrix}\Theta_\pm^{-1}\Theta=\begin{pmatrix}\widetilde A&0\\0&\widetilde B\end{pmatrix}\quad\text{on}\quad X_\pm,
\]i.e.,
\begin{equation*}\label{11.1.16}
\Theta^{-1}_{}\Theta_\pm^{}\begin{pmatrix}\widetilde A&0\\0&\widetilde B\end{pmatrix}=\begin{pmatrix}\widetilde A&0\\0&\widetilde B\end{pmatrix}\Theta^{-1}\Theta_\pm^{}\quad\text{on}\quad X_\pm.
\end{equation*}
If $C_\pm$, $D_\pm$, $E_\pm$, $F_\pm$ are the $2\times 2$ matrices with
\[
\Theta^{-1}_{}\Theta_\pm^{}=\begin{pmatrix}C_\pm&D_\pm\\E_\pm&F_\pm\end{pmatrix},
\]then this means that, on $X_\pm$,
\[C_\pm \widetilde A=\widetilde AC_\pm,\quad F_\pm \widetilde B=\widetilde B F_\pm,\quad E_\pm \widetilde A=\widetilde BE_\pm,\quad D_\pm \widetilde B=\widetilde AD_\pm,
\]i.e., for each fixed $\zeta\in U_\pm$, we have, on $\B^2$,
\begin{align*}&C_\pm(\zeta,\cdot)A= AC_\pm(\zeta,\cdot),\quad F_\pm(\zeta,\cdot) B=B F_\pm(\zeta,\cdot),\\
&E_\pm(\zeta,\cdot) A= BE_\pm(\zeta,\cdot),\quad D_\pm(\zeta,\cdot) B= AD_\pm(\zeta,\cdot).
\end{align*}
By Lemma \ref{5.1.16*} this yields that, for each $\zeta\in U_\pm$, there exist $\gamma_\pm(\zeta), \varphi_\pm(\zeta)\in \C$ with
\begin{equation}\label{3.9.16+}
\Theta_{}^{-1}(\zeta,0)^{}\Theta_{\pm}^{}(\zeta,0)=\begin{pmatrix}\gamma_\pm(\zeta)I_2&0\\0&\varphi_\pm(\zeta)I_2\end{pmatrix}\quad\text{for all}\quad \zeta\in U_\pm.
\end{equation} Since the maps $\Theta^{-1}_{}\Theta_\pm^{}$ are holomorphic and have invertible values on $X_\pm$, the so defined functions $\gamma_\pm$ and $\varphi_\pm$ must be holomorphic and different from zero on $U_\pm$. Moreover, by \eqref{10.1.16}, it follows from the equations \eqref{3.9.16+} that, for $\zeta\in U_+\cap U_-$,
\begin{equation*}
\Theta(\zeta,0)^{-1}_{}\begin{pmatrix}h(\zeta)I_2&0\\0&h^*(\zeta)I_2\end{pmatrix}
\Theta(\zeta,0)=\begin{pmatrix}\gamma^{}_+(\zeta,0)\gamma^{}_-(\zeta,0)^{-1}I_2&0\\0&\varphi^{}_+(\zeta)\varphi_-^{}(\zeta,0)^{-1}I_2
\end{pmatrix}.
\end{equation*}In particular, for each $\zeta\in U_+\cap U_+$, the matrices
\[\begin{pmatrix}h(\zeta)I_2&0\\0&h^*(\zeta)I_2\end{pmatrix}\quad\text{and}\quad
\begin{pmatrix}\gamma^{}_+(\zeta,0)\gamma^{}_-(\zeta,0)^{-1}&0\\0&\varphi^{}_+(\zeta)\varphi_-^{}(\zeta,0)^{-1}
\end{pmatrix}
\] are similar; hence they have the same eigenvalues. In particular,
\begin{equation}\label{13.1.16-}
h(\zeta)\in\Big\{\gamma^{}_+(\zeta,0)\gamma^{}_-(\zeta,0)^{-1},\varphi^{}_+(\zeta)\varphi_-^{}(\zeta,0)^{-1}\Big\}\quad\text{if}\quad \zeta\in U_+\cap U_-.
\end{equation} Consider the open sets
\[V\gamma:=\big\{\zeta\in U_+\cap U_-\,\vert\, h(\zeta)=\gamma^{}_+(\zeta,0)\gamma^{}_-(\zeta,0)^{-1}\big\}
\] and
\[V_\varphi:=\big\{\zeta\in U_+\cap U_-\,\vert\, h(\zeta)=\varphi^{}_+(\zeta,0)\varphi^{}_-(\zeta,0)^{-1}\big\}.
\]Then, by \eqref{13.1.16-}, $V_\gamma\cup V_\varphi= U_+\cap U_-$. Hence, at least one of these sets, say $V_\gamma$, is non-empty. Since the functions $h$ and  $\gamma^{}_+(\cdot,0)\gamma^{}_-(\cdot,0)^{-1}$  both are holomorphic on $U_+\cap U_-$ and $U_+\cap U_-$ is connected, it follows that $h(\zeta)=\gamma^{}_+(\zeta,0)\gamma^{}_-(\zeta,0)^{-1}$ for all $\zeta\in U_+\cap U_-$, which is not possible, by Lemma \ref{8.1.16} (ii).
\end{proof}

\end{document}